\documentclass[11pt, oneside, a4paper]{article}
\usepackage[utf8]{inputenc}
\usepackage[T1]{fontenc}
\usepackage{amsmath}
\usepackage{amsthm,enumerate} 
\usepackage{amssymb}
\usepackage{listings}
\usepackage{url}
\usepackage{hyperref}
\usepackage[T1]{fontenc}
\usepackage{selinput}
\usepackage{empheq} 
\usepackage[utf8]{inputenc}
\usepackage[T1]{fontenc}
\usepackage{amsmath}
\usepackage{amsthm,enumerate} 
\usepackage{amssymb}
\usepackage{listings}
\usepackage{url}
\usepackage{hyperref}
\usepackage{footmisc}
\usepackage{xcolor}

\newtheorem{thm}{Theorem}[section]
\newtheorem{lemma}[thm]{Lemma}
\newtheorem{theorem}[thm]{Theorem}
\newtheorem{cor}[thm]{Corollary}

\newtheorem{remark}[thm]{Remark}
\newtheorem{de}[thm]{Definition}

\newtheorem{example}[thm]{Example}

\newcommand{\philipp}[1]{\textcolor{red}{#1}}

\newcommand{\EKR}{\mathcal{F}}

\title{On the Erd\H{o}s-Ko-Rado problem of flags with type $\{1, n-3 \}$ of finite sets}
\author{Philipp Heering\\ 
\ \\
Justus-Liebig Universität Gießen,\\
Mathematisches Institut, Arndtstraße 2,\\
 D-35392 Gießen, Germany\\philipp.heering@math.uni-giessen.de}
\date{2025}

\begin{document}

\maketitle

\begin{center}
\section*{Abstract}
\end{center}

A flag of a finite set $S$ is a set $f$ of non-empty, proper subsets of $S$, such that $X\subseteq Y$ or $Y\subseteq X$ for all $X,Y\in f$. Two flags $f_1$ and $f_2$ of $S$ are opposite if $X_1\cap X_2=\emptyset$, or $X_1\cup X_2=S$ for all $X_1\in f_1$ and $X_2\in f_2$. The set $\{|X| \mid X\in f \}$ is the type of a flag $f$. A set of pairwise non-opposite flags is an Erd\H{o}s-Ko-Rado set.
In 2022 Metsch posed the problem of determining the maximum size of all Erd\H{o}s-Ko-Rado sets of flags of type $T$ with $|T|=2$. We contribute towards this by determining the maximum size for flags of type $\{ 1,n-3\}$ for finite sets with $n$ elements. 
Furthermore we answer an open questions of Metsch regarding a small case.

 \textbf{Keywords:} Erd\H{o}s-Ko-Rado sets, Kneser graphs, independence number, shifting \\
 \textbf{MDC (202):} 
 05C69, 
 05D05, 
 05C35 

\section{Introduction}

The Erd\H{o}s-Ko-Rado (EKR) theorem \cite{OriginalEKR}, published in 1961, is a fundamental result in extremal set theory. It states that for a finite set $S$ of size $n$, the maximum intersecting family of $k$-subsets of $S$, where $2k \leq n$, has size ${n-1 \choose k-1}$. Since then, numerous generalizations and extensions of the problem have been studied across countless combinatorial structures \cite{EKR_restricted, ChrisandKaren, Karen_EKR_permutations}. One natural way to generalize results on sets is to consider vector spaces over finite fields, where $k$-subsets become subspaces of dimension $k$. Indeed it was shown in \cite{ekr_vectorspaces, hsieh} that sets behave like the limit case of the vector space analogue.  

More recently, a lot of interest has developed around EKR theorems for flags of spherical Tits buildings \cite{debeule2025largestsetsnonoppositechambers, AlgebraicApproach, IMM18, EKR_polarspace_generators} 
  Here, the natural notion of non-intersecting is opposition. For spherical Tits buildings of type $A_n$ the problem can be describe as follows.

Let $\mathbb{F}_q$ be the finite field of order $q$ and consider the vector space $\mathbb{F}_q^n$. A \emph{flag} $f$ of $\mathbb{F}_q^n$ is a set of subspaces, such that $X\leq Y$, or $Y\leq X$ for all $X,Y\in f$ and the set $\{\dim(X) \mid X\in f \}$ is called the \emph{type} of $f$. Two subspaces of $\mathbb{F}_q^n$ are \emph{opposite} if they intersect trivially or if they span the entire vector space. Two flags $f_1$ and $f_2$ of $\mathbb{F}_q^n$ are called \emph{opposite} if and only if all elements of $f_1$ are opposite to all elements of $f_2$. By $\Gamma_q(n,T)$ we denote the graph, whose vertices are the flags of type $T$ of $\mathbb{F}_q^n$, two vertices are adjacent if and only if they are opposite as flags. Maximal independent sets of $\Gamma_q(n,T)$ have been studied extensively, we give some references in Table \ref{Citations}.

\begin{table}[h]
\begin{center}
\begin{tabular}{|c c|}
 \hline
Graph & Reference \\ [0.5ex]
 \hline\hline
 $\Gamma_q(n,\{1,\ldots,n-1\})$ & \cite{AlgebraicApproach, klaus_jesse_philipp}  \\ 
 \hline\hline
  $\Gamma_q(n,\{1,n-1\})$ & \cite{pointhyperplaneflags}  \\ 
 \hline\hline
  $\Gamma_q(5,\{1,2,3,4 \})$ & \cite{heering_PG4q}  \\ 
 \hline\hline
 $\Gamma_q(4,\{1,2,3 \})$ & \cite{heering_metsch_secondmax_PG3q, MetschPG3q}  \\ 
 \hline\hline
$\Gamma_q(5,\{1,4 \})$ & \cite{EKRpointplaneflags4d}  \\ 
 \hline\hline
 $\Gamma_q(5,\{2,4 \})$ & \cite{cocliquesonlineplanein4d, DHAESELEER2022103474}   \\ 
 \hline
 $\Gamma_q(7,\{3,4 \})$ & \cite{metsch_werner_plane_solid_PG6q}   \\
 \hline
\end{tabular}
\end{center}
\caption{\label{Citations} References for $\Gamma_q(n,T)$}
\end{table}

Subsequently, the EKR problem for flags of finite sets, which can be seen as thin buildings, was introduced in \cite{EKR_finite_sets_Metsch}.
Let $S$ be a finite set with $n$ elements, w.l.o.g. we assume $S=[n]:=\{1,\ldots,n\}$. Two subsets $X$ and $Y$ of $[n]$ are called \emph{opposite} if and only if the union of $X$ and $Y$ is $[n]$, or if the meet of $X$ and $Y$ is empty. This means that they are opposite if and only if their intersection is as small as possible. 
A \emph{flag} $f$ of $[n]$ is a set of subsets of $[n]$, such that $X\subseteq Y$, or $Y\subseteq X$ for all $X,Y\in f$ and the set $\{ |X| \mid X\in f \}$ is called the \emph{type} of $f$. 
Two flags $f_1$ and $f_2$ of $[n]$ are called \emph{opposite} if and only if all elements of $f_1$ are opposite to all elements of $f_2$. 
By $\Gamma(n,T)$ we denote the graph, whose vertices are the flags of type $T$ of $[n]$, two vertices are adjacent if and only if they are opposite as flags. If $|T|=1$, then $\Gamma(n,T)$ is a Kneser graph. 

In contrast to the $q$-analogue situation, little research has been conducted; in fact, as far as the author is aware, \cite{EKR_finite_sets_Metsch} is the only atricle in which independent sets of $\Gamma(n,T)$ are studied.

In \cite{EKR_finite_sets_Metsch}, Metsch posed the problem of determining the independence number of $\Gamma(n,T)$ for all cases with $|T|=2$. 
 If $T=\{a,b\}$ with $a< b$, we denote $\Gamma(n,T)$ as $\Gamma(n,a,b)$
 and the independence number as $\alpha(\Gamma(n,a,b))$. 
The independence number of $\Gamma(n,a,b)$ is known if $a+b=n$, or $a,b\leq n/2$, furthermore the graphs $\Gamma(n,a,b)$ and $\Gamma(n,n-b,n-a)$ are isomorphic \cite{EKR_finite_sets_Metsch}. Some of the remaining non-trivial cases are also settled.

\begin{thm} \cite{EKR_finite_sets_Metsch}
	For $n<2b$ and $a+3b\leq 2n$ we have $\alpha(\Gamma(n,a,b))={n-1\choose b}{b\choose a}$.
\end{thm}

This was shown using a variation of the cycle method. Furthermore, using induction on $n$, the following result was obtained.

\begin{thm} \cite{EKR_finite_sets_Metsch} \label{T: metsch 1,n-2}
	For $n\geq 5$, we have $\alpha(\Gamma(n,1,n-2))={n\choose 3}+2$.
\end{thm}

We remark that the induction used in the proof of Theorem \ref{T: metsch 1,n-2} also works in a more general setting, however the required basis is out of reach for computers. 

We introduce two tools that have not been considered in the framework of flags of finite sets. In the context of extremal set theory one technique that is used in many proofs is `shifting' \cite{frankl1987shifting, Frankl2020Simple}. We generalize this technique to flags of sets. \\
In countless areas of combinatorics a notion of `weight' is used at some point, sometimes explicitly \cite{weight-code-robin, Kurz2020AGO}
and sometimes implicitly \cite{cocliquesonlineplanein4d, k-spaces_in_k-2}. In the study of the independence number for Kneser graphs on flags of spherical Tits buildings a lot of proofs use some notion of weight extensively \cite{cocliquesonlineplanein4d, klaus_jesse_philipp, heering_metsch_secondmax_PG3q}, we also apply this notion to our framework.\\
Next, we state our main result.

\begin{theorem} \label{T: main theorem}
For $n\geq 9$, we have $\alpha(\Gamma(n,1,n-3))={n\choose 4}+42$.
\end{theorem}

In Section \ref{S: examples} we give two examples of independent sets for $n\geq 9$ that have size ${n\choose 4}+42$. The values for $n<9$ are given in Section \ref{S: computer results}.

Note that our proof involves an induction, where the cases $n=9,10$ are dealt with by the computer. These computer results rely on a result of  \cite{pan_automorphgroupgeneralposgraph} on the automorphism group of $\Gamma(n,a,b)$. 
Using this, Section \ref{S: computer results} also provides an answers for a question of Metsch concerning small $n$.

The expressions for the independence number of $\Gamma(n,1,n-3)$ and $\Gamma(n,1,n-2)$ are quite similar in nature. Even though the value $42$ is not without cultural significance, the reason for the appearance of this number is not immediately apparent from our proof. Based on some calculations and the work of \cite{EKR_finite_sets_Metsch}, we conjecture for general $s\geq 2$ that $\Gamma(n,1,n-s)$ has independence number ${n\choose s+1}+ \frac{1}{s} {s^2\choose s+1}$ if $n$ is large enough compared to $s$.

\section{Preliminaries}

For the rest of the paper let $a$, $b$ and $n$ be natural numbers with $a+b<n$ and $a< \frac{n}{2}< b$. We denote the vertices of $\Gamma(n,a,b)$ as pairs $(A,B)$ with $|A|=a$ and $|B|=b$. By flags of $\Gamma(n,a,b)$, we also denote the vertices of  $\Gamma(n,a,b)$.

\begin{lemma} \label{L: opposition}
Let $(A_1,B_1)$ and $(A_2,B_2)$ be two vertices of $\Gamma(n,a,b)$.  
Then $(A_1,B_1)$ and $(A_2,B_2)$ are adjacent and hence opposite as flags if and only if the following conditions hold:
\begin{itemize}
	\item $B_1$ and $B_2$ are opposite, i.e. $B_1\cup B_2=[n]$.
	\item $A_1$ and $B_2$ are opposite, i.e. $A_1\cap B_2=\emptyset$.
	\item $A_2$ and $B_1$ are opposite, i.e. $A_2\cap B_1=\emptyset$.
\end{itemize}
\end{lemma}

\begin{proof}
	This follows immediately from the definition of opposite and the conditions $\frac{n}{2}<b$ and $a+b<n$.
\end{proof}

We  study independent sets of $\Gamma(n,a,b)$, i.e. sets of flags that are pairwise non-opposite. Two flags are non-opposite if they do not satisfy at least one condition    stated in Lemma \ref{L: opposition}. 

\begin{remark} \label{R: intersection}
	Let $(A_1,B_1)$ and $(A_2,B_2)$ be non-opposite flags.
	Then either $|B_1\cap B_2|>2b-n$, or $|A_1\cap B_2|>0$, or $|A_2\cap B_1|>0$. In any case the intersection is large enough.\\
	Assume that $|B_1\cap B_2|>2b-n$. If $(A_3,B_3)$ is another flag, such that $|B_1\cap B_3|\geq|B_1\cap B_2|$, then by transitivity $(A_1,B_1)$ and $(A_3,B_3)$ are also non-opposite. For the other cases analogue statements hold.
\end{remark}

\subsection{Weights}

First, we assign a weight to the $b$-subsets of $[n]$.

\begin{de} Let $\EKR$ be an independent set of $\Gamma(n,a,b)$. For a subset $B\subseteq [n]$, with $|B|=b$, the number of flags in $\EKR$ that contain $B$ is called the \emph{$\EKR$-weight} of $B$ or, if there is no risk of confusion, simply \emph{weight} of $B$.
\end{de}
 
 Trivially we have that the weight of $B$ is $\leq { b\choose a }$, since there are ${b\choose a}$ flags in $\Gamma(n,a,b)$ that contain $B$. When assuming that the independent set in question is maximal, the number of possible weights decreases. We remark here that we distinguish between maximal independent sets and maximum independent sets, where maximum means largest.

\begin{lemma} \label{L: weight}
Let $\EKR$ be a maximal independent set of $\Gamma(n,a,b)$ and let $B\subset [n]$ with $|B|=b$. The weight of $B$ is ${b\choose a}$ if and only if, for every flag $(A',B')\in \EKR$, one of the following conditions holds.
	\begin{itemize}
		\item $B$ and $A'$ are non-opposite, i.e. $A'\cap B\neq \emptyset$.
		\item $B$ and $B'$ are non-opposite, i.e. $B\cup B'\neq [n]$.
	 \end{itemize}
\end{lemma}

\begin{proof}
	If all flags $(A',B')$ in $\EKR$ satisfy one of the two conditions it is clear that all flags $(A,B)$ with $A\subseteq B$ and $|A|=a$ are in $\EKR$.\\
	Now let us assume that $B$ has weight ${b\choose a}$. Let $(A,B)$ be a vertex of $\Gamma(n,a,b)$ and assume that there is a flag $(A',B')$ in $\EKR$, so that $A$ and $B'$ are non-opposite, but $B$ and $A'$ as well as $B$ and $B'$ are opposite. 
	Let $A_0$ be a subset of $B$ with $|A_0|=a$, such that $A_0\cap B'=\emptyset$. Such a set exists, since $a+b<n$ and $B\cup B'=[n]$. The flag $(A_0,B)$ cannot be in $M$, since it is opposite to $(A',B')$, but this is a contradiction as $B$ has weight ${b\choose a}$. This implies that no flag in $\EKR$ has the properties of $(A',B')$. 
\end{proof}

\begin{lemma} \label{L: weight B gap}
Let $\EKR$ be a maximal independent set of $\Gamma(n,a,b)$, 
and $B\subset [n]$ with $|B|=b$. 
If the weight of $B$ is less than ${b \choose a}$, it is at most ${b\choose a}-{n-b\choose a}$.
\end{lemma}

\begin{proof}
Let $B$ be a subset of weight $<{b \choose a}$. Then, 
by Lemma \ref{L: weight}, 
there is a flag $(A',B')$ in $\EKR$, with $B\cup B'=[n]$ and $A'\cap B=\emptyset$. 
As $B\cup B'=[n]$, we have $|B'\cap B|= 2b-n$.
For any flag $(A,B)$ that is in $\EKR$ we have that $A$ contains 
at least one of the $2b-n$ elements of $B'\cap B$. There are $b-(2b-n)=n-b$ elements in $B\setminus B'$, so any flag $(A,B)$ with $A\subseteq B\setminus B'$ is not in $\EKR$. There are ${n-b\choose a}$ such flags and the statement follows.
\end{proof}

For the case $(a,b)=(1,n-3)$ we thus have that an $(n-3)$-set has weight $n-3$ or at most $n-6$.

We can also assign weights to the $a$-subsets of $[n]$. 
\begin{de}
Let $\EKR$ be an independent set of $\Gamma(n,a,b)$. For $A\subseteq [n]$ with $|A|=a$,  the number of flags in $\EKR$ that contain $A$ is called the \emph{$\EKR$-weight} of $A$ or, if there is no risk of confusion, simply the \emph{weight} of $A$.
\end{de}


We prove an analogue of the Lemmata \ref{L: weight} and \ref{L: weight B gap} only in the special case $(a,b)=(1,n-3)$.

\begin{lemma} \label{L: weight-A}
Let $\EKR$ be a maximal independent set of $\Gamma(n,1,n-3)$, 
and $A\subseteq [n]$ with $|A|=1$. 
The subset $A$ has weight ${n-1 \choose 3}$ if and only if we have $A\subseteq B'$ for all flags $(A',B')$ of $\EKR$.
If the weight of $A$ is less than  
${n-1 \choose 3}$, it is at most ${n-1 \choose 3}-{n-4\choose 2}$. 
\end{lemma}

\begin{proof}
Assume that there is a flag $(A',B')$ in $\EKR$ with $A\cap B'=\emptyset$. The number of flags $(A,\tilde{B})$ that are opposite to $(A',B')$ can be counted as follows. There are $2$ elements in $[n]$ that are not in $A\cup B'$. Therefore $\tilde{B}$ has to contain these two elements, and also $A$. Hence we have to choose $(n-3)-3$ elements. We can choose out of all the remaining $n-3$ elements of $[n]$ except the one in $A'$. This yields that there are ${n-3-1 \choose n-3-3}={n-4\choose 2}$ flags that contain $A$ and are opposite to $(A',B')$.
\end{proof}

\subsection{Shifting}

The original Erd\H{o}s-Ko-Rado Theorem was proven with an operation on set systems that is known as shifting. 
The shifting-operator is defined as follows.

\begin{de} \label{D: shift}
Let $1 \leq i, j\leq n$ and $A\subseteq [n]$. The $i,j$-shift $S_{i,j}(A)$ of $A$ is defined as follows.
  \begin{empheq}[left={S_{i,j}(A):=\empheqlbrace}]{alignat*=2}
    & (A\setminus\{i\})\cup j \qquad    && \text{if}\  i\in A \ \text{and}\ j\notin A, \\
    & A \qquad && \text{else.}
  \end{empheq}	
\end{de}

With this we can define shifting for flags and sets of flags of $[n]$. We write the vertices of $\Gamma(n,T)$ as tuples.

\begin{de} \label{D: shift of flag}
	Let $1 \leq i, j\leq n$ and let $f=(F_{1},\ldots ,F_{t} )$ be a vertex of $\Gamma(n,T)$ for some $T\subseteq [n]$ with $|T|=t$. We define the $i,j$-shift of $f$ as follows.
	$$ S_{i,j}(f):=( S_{i,j}(F_{1}), \ldots, S_{i,j}(F_{t})  ). $$
	If $\EKR$ is a set of vertices of $\Gamma(n,T)$ and $f\in \EKR$, we define
	 \begin{empheq}[left={S_{i,j}^\EKR(f):=\empheqlbrace}]{alignat*=2}
    & S_{i,j}(f) \qquad    && \text{if}\  S_{i,j}(f) \notin \EKR, \\
    & f \qquad && \text{else.}
  \end{empheq}
and
	$$ S_{i,j}(\EKR):=\{ S_{i,j}^\EKR(f) \mid f\in \EKR \}. $$
\end{de}

It is easy to observe that for $A\subseteq B\subseteq [n]$, we have  $S_{i,j}(A)\subseteq S_{i,j}(B)\subseteq [n]$, so shifts of flags are flags.
Note that $S_{i,j}$ is a cardinality-preserving map on the power set 
of $[n]$. Therefore, the $i,j$-shift of any flag of type $T$ is also a flag of type $T$.

\begin{lemma} \label{L: S(f) and S(f') non-opp}
	If $f$ and $f'$ are flags of $\Gamma(n,T)$ that are non-opposite, then $S_{i,j}(f)$ and $S_{i,j}(f')$ are also non-opposite.
\end{lemma}

\begin{proof}
	If $f$ and $f'$ are non-opposite, there are sets $A\in f$ and $A'\in f'$ such that $A$ and $A'$ are non-opposite. 
	If $|A|+|A'|<n$, we have $A\cap A'\neq \emptyset$, and if $|A|+|A'|\geq n$, we have $A\cup A' \neq [n]$. In any case $|A\cap A'|$ is large enough. However $|A\cap A'|\leq |S_{i,j}(A)\cap S_{i,j}(A')|$. Therefore $f$ and $f'$ are also non-opposite.
\end{proof}

Note that if $f$ and $f'$ are flags of $[n]$ with type $T$ that are opposite, then $S_{i,j}(f)$ and $S_{i,j}(f')$ are not necessarily opposite. Take for example $n=4$ and $T=\{2 \}$. Then the flags $(\{1,2 \})$ and $(\{3,4 \})$ are opposite, but $S_{3,2}((\{1,2 \}))=(\{ 1,2\})$ and $S_{3,2}((\{3,4 \}))=(\{2,4 \})$ are not opposite.

\begin{lemma}
Let $\EKR$ be an independent set of $\Gamma(n,T)$ for some $T\subseteq [n]$. Then $S_{i,j}(\EKR)$ is also an independent set of $\Gamma(n,T)$ with $|\EKR|=|S_{i,j}(\EKR)|$.	
\end{lemma}

\begin{proof}
We use $S$ to denote $S_{i,j}$.
    First, let us assume that $S(\EKR)$ is not an independent set. 
	Then, we can find two flags $f$ and $f'$ in $\EKR$ with $S^\EKR(f)$ and $S^\EKR(f')$ being opposite. Since $f$ and $f'$ are non-opposite, Lemma \ref{L: S(f) and S(f') non-opp} implies that $S(f)$ and $S(f')$ are non-opposite, therefore we can assume w.l.o.g. that $S^\EKR(f)=f$, since $S(f)$ is already in $\EKR$, and $S^\EKR(f')=S(f')\neq f'$. 
	
	Let $A\in f$ and $A'\in f'$ be such that $A$ and $A'$ are non-opposite. As, by assumption $S(f')$ and $f$ are opposite, we get that $S(A')$ and $A$ are opposite. 
	In view of Remark \ref{R: intersection} this implies $|A\cap S(A')|< |A\cap A'|$.
	It follows that $i\in A\cap A'$, as well as $j\notin A\cup A'$. Hence, we have $|A\cap S(A')|=|S(A)\cap A'|$ which implies that $S(A)$ and $A'$ are opposite.

	Now, let $A\in f$ and $A'\in f'$ be such that $A$ and $A'$ are opposite. 
    Assume that $S(A)$ and $A'$ are non-opposite. This implies 
    $|S(A)\cap A'|>|A\cap A'|$ and therefore
    $i\in A$ and $j\notin A$ as well as $j\in A'$ and $i\notin A'$.
     Since $f'\neq S(f')$ there is a subset $B'$ in $f'$ with $j\notin B'$ and $i\in B'$. 
     This implies $A'\not\subseteq B'$ and $B'\not\subseteq A'$, which is a contradiction.
    Hence, $S(A)$ and $A'$ are opposite.\\
    We have shown for any $A\in f$ and $A'\in f'$  that $S(A)$ and $A'$ are opposite. This implies that $S(f)$ and $f'$ are opposite, which is a contradiction since both are in $\EKR$.

    We have shown that $S(\EKR)$ is an independent set, now we show that $|\EKR|=|S(\EKR)|$, by showing that $S^\EKR$ is injective.
    Let $f,f'$ be in $\EKR$. 
     If  $S^\EKR (f)=f$ and $S^\EKR(f')=f'$ we have $S^\EKR(f)=S^\EKR (f')$ if and only if $f=f'$.
    If  $S^\EKR (f)=S(f)\neq f$ and $S^\EKR(f')=f'$ we cannot have $S^\EKR(f)= S^\EKR(f')$ via the definition of $S^\EKR$.\\
    Now, assume $S^\EKR (f)=S(f)\neq f$ and $S^\EKR(f')=S(f')\neq f'$ with $S(f)=S(f')$. 
    Let $f=(F_{1},\ldots ,F_{t} )$ and $f'=(F'_{1},\ldots ,F'_{t} )$.
First, consider the case $S(F_k)=F_k$ and $S(F'_k)=F'_k$ for some $1\leq k\leq t$. Then clearly $F_k=F'_k$, as $S(F_k)=S(F'_k)$.
Next, consider the case $S(F_k)\neq F_k$ and $S(F'_k)\neq F'_k$ for some $1\leq k\leq t$. This yields $i\in F_k,F'_k$ and $j\notin F_k,F'_k$. As $S(F_k)=S(F'_k)$, we have $F_k=F'_k$.\\
Finally, consider the case $S(F_k)=F_k$ and $S(F'_k)\neq F'_k$ for some $1\leq k\leq t$. This yields $i\in F'_k$ and $j\notin F'_k$. As $F_k=S(F_k)=S(F'_k)$, we have $j\in F_k$ and  $i\notin F_k$. For any $\ell$ with $k<\ell\leq t$, we have $j\in F_{\ell}$, as $F_k\subseteq F_\ell$, and hence $S(F_\ell)=F_\ell$. For any $r$ with $1\leq r<k$, we have $i\notin F_r$, as $F_r\subseteq F_k$, and hence $S(F_r)=F_r$. However, this yields $S(f)=f$, which is a contradiction.
\end{proof}

A shift $S_{i,j}$ with $i\geq j$ shall be called a left-shift.
If $\EKR$ is an independent set of $\Gamma(n,T)$ and $S_{i,j}(\EKR)=\EKR$ for all $i\geq j$, we call $\EKR$ left-shifted. For any independent set $\EKR$, there is a finite number of pairs $(i_1,j_1),\ldots,(i_r,j_r)$ such that $S_{i_r,j_r}(\ldots (S_{i_1,j_1}(\EKR))\ldots )$ is a left-shifted independent set. We use this in Section \ref{Section: The proof of Theorem}, where we consider only left-shifted maximum independent sets in the proofs.

From the definition of shifting on independent sets we immediately get the following.

\begin{lemma}y \label{L: left-shift}
	Let $\EKR$ be a left-shifted independent set of $\Gamma(n,T)$ and let $f$ be a flag of $\EKR$. 
	For every left-shift $S_{i,j}$, we have that $S_{i,j}(f)$ is in $\EKR$.
\end{lemma}

\section{Examples} \label{S: examples}

The following examples for large independent sets of $\Gamma(n,a,b)$ are taken from Section 3 of \cite{EKR_finite_sets_Metsch}.

\begin{example} (\cite{EKR_finite_sets_Metsch})
For  $0\leq i \leq 2b-n$ let
	$\EKR_i(n,a,b)$ be the independent set of flags $(A,B)$ that satisfy at least one of the following two conditions
	\begin{itemize}
		\item $[i]\subseteq B \subseteq [n-1]$,
		\item $min(A)\leq i$ and $[min(A)]\subseteq B$.
	\end{itemize}
\end{example}

\begin{remark} \label{R: weight of examples}
	For $i=0$ all $B\subseteq [n]$ with $|B|=b$ that are contained in a flag of $\EKR_0(n,a,b)$ have maximal weight in the sense of Lemma \ref{L: weight}. For $i>0$ this is not the case. 
Furthermore, $\EKR_0(n,a,b)$ is left-shifted.
Finally, if $i=0$ the subsets $A\subseteq [n]$ with $|A|=a$ do not have maximal weight in the sense of Lemma \ref{L: weight-A}. For $a=1$ and $i>0$, we have that $\{ 1\}$ has maximal weight.
\end{remark}

Some of the examples above meet the bound of Theorem \ref{T: main theorem}.

\begin{remark} \label{E: example s=3}
Lemma 3.4 in \cite{EKR_finite_sets_Metsch} implies that
 the set $\EKR_{n-9}(n,1,n-3)$ has size ${n\choose 4}+42$ for all $n\geq 9$
 and that $\EKR_{n-10}(n,1,n-3)$ has size ${n\choose 4}+42$ for all $n\geq 10$. 
\end{remark}

\section{Computational results} \label{S: computer results}

If $n$ is small enough we can compute the independence number of $\Gamma(n,a,b)$. Our code is implemented in GAP4 \cite{GAP4} and uses the package GRAPE \cite{GRAPE4.9.0}. The code is available at \cite{GAPcode}. In \cite{pan_automorphgroupgeneralposgraph} it was shown that the automorphism group of $\Gamma(n,a,b)$ is $Sym(n)$, answering a question of \cite{EKR_finite_sets_Metsch}. We utilize this result in \cite{GAPcode} and obtain the following values.

\begin{table}[h]
\begin{center}
\begin{tabular}{|c c|}
 \hline
$n$ & $\alpha(\Gamma(n,1,n-3))$ \\ [0.5ex]
 \hline\hline
 5 & 8  \\ 
 \hline\hline
 6 & 30  \\ 
 \hline\hline
 7 & 60  \\ 
 \hline\hline
 8 & 105   \\ 
 \hline
 9 & 168   \\
 \hline
 10 & 252   \\
 \hline
\end{tabular}
\end{center}
\caption{\label{independence number} Independence number of $\Gamma(n,1,n-3)$}
\end{table}

For $n\leq 6$ we have $n-3\leq n/2$, so the results are somewhat trivial, for the sake of completeness we record them anyway.
We remark that $\alpha(\Gamma(n,1,n-3))=|\EKR_0(n,1,n-3)|={n-1\choose 2}\cdot (n-3)$ for $6\leq n\leq 10$. However for $n>10$, we have $|\EKR_0(n,1,n-3)|<{n\choose 4}+42$.

Note that for $n<9$, we have $\alpha(\Gamma(n,1,n-3))\neq {n\choose 4}+42$. It is also noteworthy that the runtime of our code is substantial for $n\geq 9$. 
We remark that the induction step used to prove Theorem \ref{T: metsch 1,n-2} also works for $\Gamma(n,1,n-3)$, however in order to use this step we would have to determine $\alpha(\Gamma(n,1,n-3))$ for all $n\leq171$ and this is well out of reach for any computer. 
Finally, we mention that determining the independence number of $\Gamma(9,1,6)$ was part of Problem 2 in \cite{EKR_finite_sets_Metsch}.

\section{Maximum EKR-sets of type $\{1,n-3\}$ }
\label{Section: The proof of Theorem}

For the proof of Theorem \ref{T: main theorem} we need the following two, slightly technical lemmata. The proofs involve a version of the inclusion–exclusion principle.

\begin{lemma} \label{L: technical weight}
	Let $\EKR$ be an independent set of $\Gamma(n,1,n-3)$ with $(A_1,B_1), (A_2,B_2)\in \EKR$ such that $A_1=A_2$,  $B_1\neq B_2$ and $|[n]\setminus (B_1\cup B_2) |= 2$. If $A\subseteq [n] \setminus (B_1\cup B_2)$ with $|A|=1$, then $A$ has weight at most ${n-1\choose 3}-(n-5)(n-4)+{n-5\choose 2}$.
\end{lemma}

\begin{proof}
	We count the number of flags containing $A$ that are opposite to $(A_1,B_1)$ or $(A_2,B_2)$. 
	The number of flags containing $A$ that are opposite to one of those is, according to Lemma \ref{L: weight-A} exactly ${n-4\choose 2}$ and $2{n-4\choose 2}=(n-5)(n-4)$.
Now, we count the number of flags containing $A$ that are opposite to both $(A_1,B_1)$ and $(A_2,B_2)$.
	
	Such a flag $(A,B)$ must satisfy $A_1\cap B=\emptyset$ as well as $B\cup B_1=B\cup B_2=[n]$.
	So $[n]\setminus (B_1\cap B_2)$ has to be a subset of $B$.
	We have $|B_1|=|B_2|=n-3$ and $|B_1\cup B_2|=n-2$ as a premise. This yields $|B_1\cap B_2|=n-4$ and $|[n]\setminus (B_1\cap B_2)|=4$.
Therefore, $4$ elements of $B$ are fixed, and keeping in mind that $A_1\cap B=\emptyset$ we can choose the remaining $n-3-4$ elements freely out of the remaining $n-4-1$ elements. This yields ${n-4-1 \choose n-3-4}={n-5\choose 2}$ flags and the statement follows. 	 
\end{proof}

\begin{lemma} \label{L: technical weight 2}
	Let $\EKR$ be an independent set of $\Gamma(n,1,n-3)$.
	Let $B_1$ and $B_2$ be two distinct subsets of $[n]$ with size $n-3$, such that the $\EKR$-weight of $B_i$ is $n-3$ for $i=1,2$.
\begin{enumerate} [(a)]
    \item If $A\subseteq [n] \setminus B_1$ with $|A|=1$, then $A$ has weight at most ${n-1\choose 3}-{n-3\choose 3}$.
    \item If $|[n]\setminus (B_1\cup B_2) |= 2$ and $A\subseteq [n] \setminus (B_1\cup B_2)$ with $|A|=1$, then $A$ has weight at most ${n-1\choose 3}-2{n-3\choose 3}+{n-4\choose 3}$.
\end{enumerate}
\end{lemma}

\begin{proof}
	First, we count the number of flags $(A,B)$ that are opposite to some flag containing $B_1$. As $A\cap B_1=\emptyset$ and $B_1$ has maximal weight, this means $B\cup B_1=[n]$. There are three elements in $[n]\setminus B_1$ and these elements have to be in $B$. Now, we can choose the remaining $n-3-3$ elements of $B$ freely out of the $n-3$ elements in $B_1$. In conclusion we have that the number in question is ${n-3\choose n-3-3}={n-3 \choose 3}$. This implies (a).
	
	Now, we count the number of flags $(A,B)$ that are opposite to some flag containing $B_1$ and some flag containing $B_2$.
	 As $A\cap (B_1\cup B_2)=\emptyset$ and $B_1$ and $B_2$ have maximal weight, this means $B\cup B_1=B\cup B_2=[n]$.
	 We have $|B_1|=|B_2|=n-3$.
     As  $|[n]\setminus (B_1\cup B_2) |= 2$, we have $|B_1\cup B_2|=n-2$.
     This yields $|B_1\cap B_2|=n-4$, hence $|[n]\setminus (B_1\cap B_2) |= 4$. So four elements in $B$ are fixed.
	   Now we can choose the remaining $n-3-4$ elements freely out of the $n-4$ elements in $B_1\cap B_2$. In conclusion we have that the number in question is ${n-4\choose n-3-4}={n-4 \choose 3}$.
       By the inclusion-exclusion principle, the statement of (b) follows.
\end{proof}

\begin{lemma} \label{L:ABproperties}
    Let $\EKR$ be a left-shifted independent set of $\Gamma(n,1,n-3)$. Let $\EKR_2$ be the set of flags $(A,B)$ in $\EKR$ with $1\notin B$. If $|\EKR_2|>(n-3)$, we have that $\EKR_2$ contains two flags $(A_1,B_1), (A_2,B_2)$ with $B_1\neq B_2$,  as well as $n\notin B_1\cup B_2$ and $A_1=A_2=\{2\}$. 
\end{lemma}

\begin{proof}
    As $|\EKR_2|>(n-3)$, we have that $\EKR_2$ contains two flags $(A_1,B_1), (A_2,B_2)$ with $B_1\neq B_2$. We show that we can assume w.l.o.g. that $n\notin B_1\cup B_2$ as well as $A_1=A_2=\{2\}$.

If $A_1=\{i\}$ and $A_2=\{j\}$ with $2\in B_1,B_2$, and $n\notin B_1,B_2$ we can consider $S_{i,2}((A_1,B_1))$ and $S_{j,2}((A_2,B_2))$ instead of $(A_1,B_1)$ and $(A_2,B_2)$. As $i,j\geq 2$ we have that $S_{i,2}((A_1,B_1))$ and $S_{j,2}((A_2,B_2))$ are in $\mathcal{F}$, since $\mathcal{F}$ is left-shifted.\\
This implies that we just have to show that we can assume w.l.o.g. that $n\notin B_1,B_2$ and $2\in B_1,B_2$. 

First, we show that we can assume $n\notin B_1,B_2$. 

Assume that $n\notin B_1$ but $n\in B_2$. 
As $|B_2|=n-3$ and $1\notin B_2$ there are two integers $i$ and $j$  not in $B_2$ such that $i,j>1$. For these integers $i$ and $j$ we have $S_{n,i}(B_2)\neq S_{n,j}(B_2)$, so w.l.o.g. $S_{n,i}(B_2)\neq B_1$.
As $n>i$ we have $S_{n,i}((A_2,B_2))\in \mathcal{F}$, since $(A_2,B_2)\in\mathcal{F}$ and $\mathcal{F}$ is left-shifted.
So in this case we consider $(A_1,B_1)$ and $S_{n,i}((A_2,B_2))$, instead of $(A_1,B_1), (A_2,B_2)$. 

Assume that $n\in B_1,B_2$. As $|B_1|=n-3$ and $1\notin B_1$ there is an integer $k>1$ that is not in $B_1$ and we consider $S_{n,k}(B_1)$. As $|B_2|=n-3$ and $1\notin B_2$ there are two integers $i$ and $j$  not in $B_2$ such that $i,j>1$. For these integers $i$ and $j$ we have $S_{n,i}(B_2)\neq S_{n,j}(B_2)$, so w.l.o.g. $S_{n,i}(B_2)\neq S_{n,k}(B_1)$. As $n>i,k$ we have $S_{n,k}((A_1,B_1)),S_{n,i}((A_2,B_2))\in \mathcal{F}$, since $(A_1,B_1),(A_2,B_2)\in\mathcal{F}$ and $\mathcal{F}$ is left-shifted.
So in this case we consider $S_{n,k}((A_1,B_1))$ and $S_{n,i}((A_2,B_2))$, instead of $(A_1,B_1), (A_2,B_2)$. 

Now, we show that we can also assume $2\in B_1,B_2$.

Assume that $2\in B_1$, but $2\notin B_2$.
For $i,j\in B_2$ with $i\neq j$, we have that $S_{i,2}(B_2)\neq S_{j,2}(B_2)$, so without loss of generality $S_{i,2}(B_2)\neq B_1$. 
As $i>2$ and $\mathcal{F}$ is left-shifted, we have $S_{i,2}((A_2,B_2))\in \mathcal{F}$.
So in this case we consider $(A_1,B_1)$ and $S_{i,2}((A_2,B_2))$, instead of $(A_1,B_1), (A_2,B_2)$.

Assume that $2\notin B_1,B_2$.
For some $k\in B_1$ we consider $S_{k,2}(B_1)$, so $2\in S_{k,2}(B_1)$.
For $i,j\in B_2$ with $i\neq j$, we have that $S_{i,2}(B_2)\neq S_{j,2}(B_2)$, so w.l.o.g. $S_{i,2}(B_2)\neq S_{k,2}(B_1)$.
So in this case we consider $S_{k,2}((A_1,B_1))$ and $S_{i,2}((A_2,B_2))$, instead of $(A_1,B_1), (A_2,B_2)$. 
\end{proof}

The following lemma is the key to enable an inductive proof of Theorem \ref{T: main theorem}.


\begin{lemma} \label{L: key for the proof}
	Let $n\geq 11$ and let $\EKR$ be a maximum left-shifted independent set of $\Gamma(n,1,n-3)$. 
Then $\{1\}$ has $\EKR$-weight ${n-1 \choose 3}$. 
\end{lemma}

Before we start with the proof, we remark that $n\geq 11$ is a necessary condition in order for the statement to be true. For $n=9,10$ we can find maximum left-shifted independent set of size ${n\choose 4}+42$, where $\{1\}$ does not have maximal weight, see Remark \ref{R: weight of examples} and Remark \ref{E: example s=3}.

\begin{proof}
 Let us assume that the $\EKR$-weight of $\{1\}$ is smaller than ${n-1 \choose 3}$. Lemma \ref{L: weight-A} states that the weight of $\{1\}$ is ${n-1 \choose 3}$ if and only if all flags $(A,B)$ of $\EKR$ satisfy $1\in B$. So in particular $\EKR$ contains flags $(A,B)$ with $1\notin B$.
 
  Let $\EKR_1$ be the set of flags $(A,B)$ with $A=\{1\}$ that are not in $\EKR$. Furthermore, let $\EKR_2$ be the set of flags $(A,B)$ in $\EKR$ with $1\notin B$. 
 If $|\EKR_2|<|\EKR_1|$ we would have that $(\EKR \setminus \EKR_2)\cup \EKR_1$ is an independent set of size bigger than $|\EKR|$. This stands in contradiction to $|\EKR|$ being a maximum independent set. Therefore, we have 
 \begin{align} \label{A: basic}
 	|\EKR_2|\geq |\EKR_1|.
 \end{align}
 
 Lemma \ref{L: weight-A} yields $|\EKR_1|\geq {n-4\choose 2}$. As $|\EKR_2|\geq |F_1|\geq {n-4\choose 2}$, we have $|\EKR_2|> (n-3)$ for $n\geq 8$.
 Lemma \ref{L:ABproperties} implies that $\EKR_2$ contains flags $(A_1,B_1)$, $(A_2,B_2)$ with $B_1\neq B_2$, as well as $A_1=A_2=\{2\}$ and $n\notin B_1\cup B_2$.
 As the flags are in $\EKR_2$, we have $1\notin B_1\cup B_2$. If there would be a third integer not in $B_1\cup B_2$, we would have $B_1=B_2$. Therefore, $|[n]\setminus (B_1\cup B_2)|=2$ and we can apply Lemma \ref{L: technical weight} and obtain
 \begin{align} \label{A: EKR_1^c ineq 1}
 	|\EKR_1|\geq (n-5)(n-4)-{n-5\choose 2}.
 \end{align}

Now, assume that $(A,B)$ and $(A',B')$ are flags in $\EKR_2$ such that $B\cup B'=[n]\setminus \{1\}$. We show that this leads to a contradiction by showing that there are left shifts of $(A,B)$ and $(A',B')$ in $\EKR$ that are opposite. 
Let $A=\{i\}$ and $A'=\{j\}$. 

If $j\in B$ and $i\notin B'$, we can consider $(A'',B'')=S_{j,1}((A',B'))$. Then $A''=\{1\}\not\subseteq B$ and $A=\{i\}\not\subseteq B''$. Furthermore $i,j\in B$ and $1\in B''$, so $B\cup B''=[n]$. Therefore, $(A,B)$ and $(A'',B'')$ are opposite.
If $i\in B'$ and $j\notin B$ we can use an analogue argument.

If $j\in B$ and $i\in B'$, we can assume w.l.o.g. that $i\geq j$. Then we can consider  $(A'',B'')=S_{i,j}(S_{j,1}((A',B')))$. Then $i\notin B''$ and $1\notin B$. As $i,j\in B$ and $1\in B''$, we have $B\cup B''=[n]$. Therefore, $(A,B)$ and $(A'',B'')$ are opposite.

If $j\notin B$ and $i\notin B'$, we can assume w.l.o.g. that $i> j$. As $|B\cup B'|=n-1$ and $|B|=|B'|=n-3$, we have $|B\cap B'|=n-5$. As $n\geq 11$ there is a $k\in B\cap B'$ and $k\neq i,j$. Now, we can consider  $(A'',B'')=S_{k,1}((A',B'))$. Then $A''=\{j\}\not\subseteq B$ and  $A=\{i\}\not\subseteq B''$. Furthermore, $i\in B$ and $1,j\in B''$, so $B\cup B''=[n]$. Hence, $(A,B)$ and $(A'',B'')$ are opposite.

We have shown that no two flags $(A,B)$ and $(A',B')$ in $\EKR_2$ can satisfy $B\cup B'=[n]\setminus \{1\}$. 
If $(A,B)\in \EKR_2$, the complement $B^c$ of $B$ in $[n]\setminus \{1\}$ is a subset of size $2$. The condition $B\cup B'\neq [n]\setminus \{1\}$ translates to $B^c\cap B'^c \neq \emptyset$. By the Erd\H{o}s-Ko-Rado theorem we get at most ${(n-1)-1 \choose 2-1}=n-2$ distinct subsets $B$ of size $n-3$ in flags of $\EKR_2$.
Let $w$ be the average weight of such a set $B$ in $\EKR_2$.
	 Using (\ref{A: basic}) and (\ref{A: EKR_1^c ineq 1})  this yields
	\begin{align*} 
		w(n-2)\geq  |\EKR_2|\geq |\EKR_1| \geq (n-5)(n-4)-{n-5\choose 2}.
	\end{align*}
So in particular $w\geq  \frac{(n-5)(n-4)-(n-6)(n-5)/2}{n-2}=\frac{n-5}{2}$.

First, assume that no $(n-3)$-set of a flag in $\EKR_2$ has maximal weight. 
In this case Lemma \ref{L: weight B gap} implies that the sets $B$ of flags in $\EKR_2$ have weight at most $n-6$.
We also have that there is a flag $(A',B')$ in $\EKR_2$, for which $B'$ has weight $w\geq \frac{n-5}{2}$.
We count how many flags $(A,B)$ with $A=\{1\}$ are opposite to $(A',B')$. Let $W\subseteq B'$ with $|W|=w$ such that all flags $(A',B')$ with $|A'|=1$ and $A'\subseteq W$ are in $\EKR_2$. 
The flags $(A,B)$ with $A=\{1\}$ that are opposite to at least one flag containing $B'$ are the flags that satisfy $B\cup B'=[n]$ and $W\not\subseteq B$. 
First, we count how many flags satisfy $B\cup B'=[n]$.
We have $([n]\setminus B')\subseteq B$, so $3$ elements of $B$ are fixed and there are $n-3-3$ left to choose out of $n-3$, so the number in question is ${n-3\choose 3}$. Now, we count how many of those satisfy $W\subseteq B$. In this case $w+3$ elements of $B$ are fixed and there are $(n-3)-(w+3)$ left to choose out of $n-(w+3)$, so the number in question is ${n-w-3 \choose 3}$.
As $w\geq \frac{n-5}{2}$, we have ${n-3\choose 3}-{n-w-3 \choose 3}\geq {n-3\choose 3}-{\frac{n-1}{2}\choose 3}$. In view of (\ref{A: basic}), this gives
$$ (n-2)(n-6) \geq  |\EKR_2|\geq |\EKR_1| \geq  {n-3\choose 3}-{\frac{n-1}{2}\choose 3} $$
and this is a contradiction for $n\geq 11$.

Now, assume that there is exactly one set $B$ in a flag of $\EKR_2$ that has maximal weight.
We can apply Lemma \ref{L: technical weight 2} (a) which yields $|\EKR_1|\geq {n-3\choose 3}$. 
On the other hand Lemma \ref{L: weight B gap} implies that the remaining sets $B$ of flags in $\EKR_2$ (there are at most $n-3$ such sets) have weight at most $n-6$. Using (\ref{A: basic}), this yields
$$ (n-3)+(n-3)(n-6) \geq  |\EKR_2|\geq |\EKR_1| \geq  {n-3\choose 3} $$
and this is a contradiction for $n\geq 11$.

Finally, assume that there are at least two $(n-3)$-sets $B_1$ and $B_2$ in  flag of $\EKR_2$ that have maximal weight. 
We have already seen that it is not possible that $B_1\cup B_2=[n]\setminus \{1\}$. Hence, $|[n]\setminus (B_1\cup B_2) |= 2$.
So, we can apply Lemma \ref{L: technical weight 2} (b) which implies $|\EKR_1|\geq 2{n-3\choose 3}-{n-4\choose 3}$. Using (\ref{A: basic}), this yields
$$ (n-2)(n-3) \geq  |\EKR_2|\geq |\EKR_1| \geq   2{n-3\choose 3}-{n-4\choose 3} $$
and this is a contradiction for $n\geq 11$.
\end{proof}

\begin{cor}
For $n\geq 11$, we have $\alpha(\Gamma(n,1,n-3))\leq {n\choose 4}+42$.
\end{cor}
\begin{proof}
Let $\EKR$ be a maximum left-shifted independent set of $\Gamma(n,1,n-3)$. 
By Lemma \ref{L: key for the proof}, the subset $\{1\}$ has weight ${n-1 \choose 3}$. Lemma \ref{L: weight-A} implies that every flag $(A,B)$ in $\EKR$ satisfies $1\in B$.
Consider the set of flags $(A,B)$ in $\EKR$ with $A\neq \{1\}$.
By considering $(A,B\setminus\{ 1\})$ for any such flag, we obtain an independent set of $\Gamma(n-1,1,n-4)$. In particular we have 
$$\alpha(\Gamma(n,1,n-3))\leq {n-1 \choose 3}+ \alpha(\Gamma(n-1,1,n-4)).$$

For $n=11$, we have $\alpha(\Gamma(n-1,1,n-4))\leq {n-1 \choose 4}+42$ by Section \ref{S: computer results}, so for general $n$ this is the induction hypothesis and we have
\begin{align*}
	\alpha(\Gamma(n,1,n-3))\leq {n-1 \choose 3}+{n-1 \choose 4}+42={n\choose 4}+42 
\end{align*}
as claimed.
\end{proof}

Together with Remark \ref{E: example s=3} and the results of Section \ref{S: computer results}, this proves Theorem \ref{T: main theorem}.

\subsection*{Acknowledgements}
  
The author would like to thank Klaus Metsch for bringing this problem to his attention.
The author would also like to thank Alexander Gavrilyuk for helpful discussions.

\bibliographystyle{plain}
\bibliography{EKR_flags_1,n-3_Heering.bib}

\begin{thebibliography}{10}

\bibitem{weight-code-robin}
Sam Adriaensen, Robin Simoens, and Leo Storme.
\newblock {The minimum weight of the code of intersecting lines in PG$(3,q)$}.
\newblock {\em Ars Mathematica Contemporanea}, 01 2025.

\bibitem{debeule2025largestsetsnonoppositechambers}
Jan~De Beule, Philipp Heering, Sam Mattheus, and Klaus Metsch.
\newblock {The largest sets of non-opposite chambers in spherical buildings of type $B$}.
\newblock {\em arXiv}, 2025.

\bibitem{EKRpointplaneflags4d}
A.~Blokhuis, A.~E. Brouwer, and T.~Sz\H{o}nyi.
\newblock Maximal cocliques in the {K}neser graph on point-plane flags in {${\rm PG}(4,q)$}.
\newblock {\em European J. Combin.}, 35:95--104, 2014.

\bibitem{cocliquesonlineplanein4d}
Aart Blokhuis and Andries~E. Brouwer.
\newblock Cocliques in the {K}neser graph on line-plane flags in {${\rm PG}(4,q)$}.
\newblock {\em Combinatorica}, 37(5):795--804, 2017.

\bibitem{pointhyperplaneflags}
Aart Blokhuis, Andries~E. Brouwer, and \c{C}i\c{c}ek G\"{u}ven.
\newblock Cocliques in the {K}neser graph on the point-hyperplane flags of a projective space.
\newblock {\em Combinatorica}, 34(1):1--10, 2014.

\bibitem{AlgebraicApproach}
Jan De~Beule, Sam Mattheus, and Klaus Metsch.
\newblock An algebraic approach to {E}rdös-{K}o-{R}ado sets of flags in spherical buildings.
\newblock {\em Journal of Combinatorial Theory, Series A}, 192:Paper No. 105657, 33, 2022.

\bibitem{k-spaces_in_k-2}
Jozefien D'Haeseleer, Giovanni Longobardi, Ago-Erik Riet, and Leo Storme.
\newblock {Maximal Sets of $k$-Spaces Pairwise Intersecting in at Least a $(k-2)$-Space}.
\newblock {\em The Electronic Journal of Combinatorics}, 29, 03 2022.

\bibitem{DHAESELEER2022103474}
Jozefien D’haeseleer, Klaus Metsch, and Daniel Werner.
\newblock {On the chromatic number of two generalized Kneser graphs}.
\newblock {\em European Journal of Combinatorics}, 101:103474, 2022.

\bibitem{OriginalEKR}
Pál Erd\H{o}s, Chao Ko, and Richard Rado.
\newblock {Intersection theorems for systems of finite sets}.
\newblock {\em Quart. J. Math. Oxford Ser. (2)}, 12:313--320, 1961.

\bibitem{frankl1987shifting}
Peter Frankl.
\newblock The shifting technique in extremal set theory.
\newblock {\em Surveys in combinatorics}, 123:81--110, 1987.

\bibitem{EKR_restricted}
Peter Frankl.
\newblock {Erdős–Ko–Rado Theorem for a Restricted Universe}.
\newblock {\em The Electronic Journal of Combinatorics}, 27, 05 2020.

\bibitem{Frankl2020Simple}
Peter Frankl and Andrey Kupavskii.
\newblock Simple juntas for shifted families.
\newblock {\em Discrete Analysis}, 2020.

\bibitem{ekr_vectorspaces}
Peter Frankl and Richard~M. Wilson.
\newblock {The Erdös-Ko-Rado theorem for vector spaces}.
\newblock {\em Journal of Combinatorial Theory, Series A}, 43(2):228--236, 1986.

\bibitem{GAP4}
The GAP~Group.
\newblock {\em {GAP -- Groups, Algorithms, and Programming, Version 4.12.2}}, 2022.

\bibitem{ChrisandKaren}
Chris Godsil and Karen Meagher.
\newblock {An algebraic proof of the Erd\H{o}s-Ko-Rado theorem for intersecting families of perfect matchings}.
\newblock {\em Ars Mathematica Contemporanea}, 12, 06 2015.

\bibitem{GAPcode}
Philipp Heering.
\newblock {{GAP4} Code for EKR problem on flags of finite sets}.
\newblock \url{https://github.com/PhilippHeering/Code_for_flags_of_finite_sets}, 2025.

\bibitem{heering_PG4q}
Philipp Heering.
\newblock {On the largest independent sets in the Kneser graph on chambers of PG(4,q)}.
\newblock {\em Discrete Mathematics}, 348(5), 2025.

\bibitem{klaus_jesse_philipp}
Philipp Heering, Jesse Lansdown, and Klaus Metsch.
\newblock {Maximum Erd\H{o}s-Ko-Rado sets of chambers and their antidesigns in vector-spaces of even dimension}.
\newblock {\em arXiv}, 2024.

\bibitem{heering_metsch_secondmax_PG3q}
Philipp Heering and Klaus Metsch.
\newblock {Maximal cocliques and the chromatic number of the Kneser graph on chambers of PG$(3,q)$}.
\newblock {\em Journal of Combinatorial Designs}, 32(7):388--409, 2024.

\bibitem{hsieh}
Wen-Ning Hsieh.
\newblock Intersection theorems for systems of finite vector spaces.
\newblock {\em Discrete Mathematics}, 12(1):1--16, 1975.

\bibitem{IMM18}
Ferdinand Ihringer, Klaus Metsch, and Bernhard M\"uhlherr.
\newblock An {EKR}-theorem for finite buildings of type {$D_{\ell}$}.
\newblock {\em J. Algebraic Combin.}, 47(4):529--541, 2018.

\bibitem{Kurz2020AGO}
Sascha Kurz and Sam Mattheus.
\newblock {A Generalization of the Cylinder Conjecture for Divisible Codes}.
\newblock {\em IEEE Transactions on Information Theory}, 68:2281--2289, 2020.

\bibitem{Karen_EKR_permutations}
Karen Meagher and Andriaherimanana~Sarobidy Razafimahatratra.
\newblock {The Erdős-Ko-Rado Theorem for 2-Pointwise and 2-Setwise Intersecting Permutations}.
\newblock {\em The Electronic Journal of Combinatorics}, 28, 10 2021.

\bibitem{MetschPG3q}
Klaus Metsch.
\newblock {The Chromatic Number of Two Families of Generalized Kneser Graphs Related to Finite Generalized Quadrangles and Finite Projective 3-Spaces}.
\newblock {\em The Electronic Journal of Combinatorics}, 28, 07 2021.

\bibitem{EKR_finite_sets_Metsch}
Klaus Metsch.
\newblock {Erdős-Ko-Rado sets of flags of finite sets}.
\newblock {\em Journal of Combinatorial Theory, Series A}, 191:105641, 10 2022.

\bibitem{metsch_werner_plane_solid_PG6q}
Klaus Metsch and Daniel Werner.
\newblock Maximal cocliques in the kneser graph on plane-solid flags in pg$(6,q)$.
\newblock {\em Innovations in Incidence Geometry: Algebraic, Topological and Combinatorial}, 18:39--55, 11 2020.

\bibitem{pan_automorphgroupgeneralposgraph}
Junyao Pan.
\newblock The full automorphism groups of general position graphs.
\newblock {\em Journal of Combinatorial Theory, Series A}, 201:105800, 2024.

\bibitem{EKR_polarspace_generators}
Valentina Pepe, Leo Storme, and Frédéric Vanhove.
\newblock {Theorems of Erdős–Ko–Rado type in polar spaces}.
\newblock {\em Journal of Combinatorial Theory, Series A}, 118(4):1291--1312, 2011.

\bibitem{GRAPE4.9.0}
L.~H. Soicher.
\newblock {GRAPE}, graph algorithms using permutation groups, {V}ersion 4.9.0.
\newblock \href {https://gap-packages.github.io/grape} {\texttt{https://gap-packages.github.io/}\discretionary {}{}{}\texttt{grape}}, Dec 2022.
\newblock Refereed GAP package.

\end{thebibliography}

\end{document}